\documentclass[12pt]{article}   	
\usepackage{geometry}                		
\usepackage[parfill]{parskip}    		
\usepackage{graphicx}				
\usepackage[pdftex,pagebackref,letterpaper=true,colorlinks=true,pdfpagemode=none,urlcolor=blue,linkcolor=blue,citecolor=blue,pdfstartview=FitH]{hyperref}
\usepackage{amsmath,amsfonts}
\usepackage{color}
\usepackage{amsthm}
\usepackage{boolexpr}
\usepackage{amssymb,latexsym,pstricks,pst-plot}
\usepackage[english]{babel}
\usepackage{pgf}
\usepackage{tikz}
\usetikzlibrary{arrows,automata}
\usepackage[latin1]{inputenc}
\usepackage[UKenglish]{isodate}
\usepackage{commath}
\allowdisplaybreaks

\newcommand{\BigO}[1]{\ensuremath{\operatorname{O}\bigl(#1\bigr)}}

\usepackage{amssymb}
\newtheorem{theorem}{Theorem}[section]
\newtheorem{lemma}[theorem]{Lemma}

\newtheorem{corollary}[theorem]{Corollary}

\title{Representing numbers as the sum of squares and powers in the ring $\mathbb{Z}_n$}
\author{Rob Burns}

\begin{document}
\maketitle
\begin{abstract}
We examine the representation of numbers as the sum of two squares in $\mathbb{Z}_n$ for a general integer $n$. Using this information we make some comments about the density of positive integers which can be represented as the sum of two squares and powers of $2$ in $\mathbb{N}$.
\end{abstract}

\section{Introduction}
\label{intro}
The problem of determining which numbers can be represented as the sum of two squares in $\mathbb{N}$ was solved to a large extent by Fermat who stated that an odd prime $p$ can be represented as the sum of two squares of integers if and only if \mbox{$\,\, p \equiv 1 \mod 4$}. Euler later provided the proof. More generally, an positive integer can be written as a sum of two squares if and only if its prime decomposition contains no prime congruent to \mbox{$\, 3  \mod 4$} raised to an odd power. A related problem is to determine the asymptotic density of the positive integers which can be represented as the sum of two squares. Landau \cite{Lan1908} and Ramanujan independently showed that, for large $x$, the number of positive integers less than $x$ that can be expressed as the sum of two squares is approximately \mbox{$\frac{K \, x}{\ln x}$} where $K$ is called the Landau-Ramanujan constant.

This paper looks at the representation of numbers as the sum of two squares in $\mathbb{Z}_n$. We allow $0$ as one (or both) of the squares in the representation. As is the case with $\mathbb{N}$, when dealing with $\mathbb{Z}_n$ we can determine exactly which numbers can and which numbers cannot be given such a representation. We can also derive an exact formula for the size of the set of numbers which can be represented in this way in $\mathbb{Z}_n$. These results are presented in section \ref{sumsofsquares}. Some of the results were previously established in \cite{Harrington:2014aa}.

One of the reasons for reducing the problem to the ring $\mathbb{Z}_n$ is that the results from $\mathbb{Z}_n$ can provide information about what happens in $\mathbb{N}$. We provide an example of this in section~\ref{powers} where we deal with the problem discussed in \cite{Crocker:2008id} and \cite{Platt:2016aa} of representing numbers as the sum of two squares and powers of $2$. We will use an argument from \cite{Platt:2016aa} to give a lower bound for the density of numbers which cannot be represented as the sum of two squares and up to one power of $2$. The argument uses the simple idea that if a number $m$ can be written as the sum of two squares in $\mathbb{N}$ then $m$ can be written as the sum of two squares modulo any positive integer. Conversely, if there is one positive integer $n$ such that $m$ cannot be written as the sum of two squares in $\mathbb{Z}_n$ then $m$ cannot be expressed as the sum of two squares in $\mathbb{N}$.

\bigskip

\section{Some preliminary results}
\label{prelim}

We first provide some notation for the sets of numbers we will be using. Let $n \in \mathbb{N}$. We define the sets $S_n$ and $N_n$ by
$$
\mbox{$S_n := \{\, m \in \mathbb{Z}_n : m = x^2 + y^2$ for some $x, y \in \mathbb{Z}_n \, \}$ }
$$
and
$$
N_n := \mathbb{Z}_n \, \backslash \, S_n.
$$
We note that Harrington, Jones and Lamarche \cite{Harrington:2014aa} use $S^0_n$ instead of $S_n$ for the set of numbers in $\mathbb{Z}_n$ which can be written as a sum of up to two (possibly zero) squares $\mod n$. If $A$ is a subset of $\mathbb{Z}_n$ then define $\abs{A}$ to be the number of elements in $A$ and  define the density $r(A)$ of $A$ by
$$
r(A) := \frac{\abs{A}}{\abs{\mathbb{Z}_n}} = \frac{\abs{A}}{n}.
$$
The rest of the results in this section are not new but have been included for completeness as they are required in the next sections.

\bigskip
\begin{lemma}
\label{monoid}
The set $S_n$ is a monoid under the operation of multiplication for each $n \in \mathbb{N}$.
\end{lemma}
\bigskip
\begin{proof}
Firstly, $1 \in S_n$ as $ 1= 1^2 + 0^2$. Secondly, $S_n$ is closed under multiplication because of the Brahmagupta - Fibonacci identity.
\end{proof}

\bigskip
\begin{lemma}
\label{relprime}
Let $m, n \in \mathbb{N}$ be relatively prime. Then 
$$
\mbox{$x \in S_{mn} \,\, $ if and only if $ \,\,  x \mod m \in S_m \, $ and $ \, x \mod n \in S_n$}.
$$
In addition
\begin{equation}
\label{abssmn}
\abs{S_{mn}} = \abs{S_m} \times \abs{S_n}
\end{equation}
and
\begin{equation}
\label{absnmn}
\abs{N_{mn}} = n \times \abs{N_m} + m \times \abs{N_n} - \abs{N_m} \times \abs{N_n}.
\end{equation}
\end{lemma}
\bigskip
\begin{proof}
The first statement follows from the Chinese remainder theorem. Equation (\ref{abssmn}) follows. Equation (\ref{absnmn}) is then a consequence of the fact that 
$$
\mbox{$\abs{N_m} = m - \abs{S_m}$ \,\, and \,\, $\abs{N_n} = n - \abs{S_n}$.} 
$$
Alternatively,
$$
\mbox{$x \in N_{mn}$ \, \, if and only if \, \, $x \mod m \in N_{m}$ \, or $x \mod n \in N_{n}$.}
$$
Now
$$
\abs{ \{ \, x \in \mathbb{Z}_{mn} : x \mod m \in N_m \, \} } = n \cdot \abs{N_m}
$$
and
\begin{equation*}
\abs{ x \in \mathbb{Z}_{mn} : x \mod m \in N_m \,\, \text{and}  \,\, x \mod n \in N_n}  = \abs{N_m} \cdot \abs{N_n}
\end{equation*}
by the Chinese remainder theorem. Equation (\ref{absnmn}) follows from the standard set theory formula for the size of the union of two sets.
\end{proof}  

\bigskip

\begin{corollary}
\label{cor}
Let $m_1, m_2, ... , m_k \in \mathbb{N} \, \, $ be relatively prime and $ \, \, n = \prod_1^k m_i$. Then 
\begin{equation}
\label{rmns}
r(S_n) = \prod_1^k r(S_{m_i})
\end{equation}
and
\begin{equation}
\label{rmnn}
r(N_n) = 1 - \prod_1^k (1 - r(N_{m_i}) \, ).
\end{equation}
\end{corollary}
\bigskip
\begin{proof}
When $k = 2$, equations (\ref{rmns}) and (\ref{rmnn})  are established by dividing equations (\ref{abssmn}) and (\ref{absnmn}) by $mn$. The general case follows by induction.
\end{proof}

\bigskip

\section{Sums of two squares in $\mathbb{Z}_n$}
\label{sumsofsquares}

The natural approach to investigating $\mathbb{Z}_n$ is to see what happens when $n$ is a power of a prime and then build up to a general $n \in \mathbb{N}$ using the Chinese remainder theorem and other devices. The behaviour of sums of squares in $\mathbb{Z}_{p^k}$ for $p$ prime and $k \geq 1$ depends on whether $p = 2$, $p \equiv 1 \mod 4$ or $p \equiv 3 \mod 4$. We deal with each case separately.

\bigskip

\begin{theorem}
\label{2k}
Let $k \in \mathbb{N}$, $k \geq 2$. Then
$$
S_{2^k} = \{ 0 \} \cup \{ t \cdot 2^s : s \in \mathbb{N}, \, \, t \equiv 1 \mod 4 \}
$$
and
$$
N_{2^k} = \{ t \cdot 2^s : s \in \mathbb{N}, \, \, t \equiv 3 \mod 4 \}.
$$
So
$$
\abs{S_{2^k}} = 2^{k-1} + 1 \, \, \, \text{and} \, \, \, \abs{N_{2^k}} = 2^{k-1} - 1.
$$
\end{theorem}
\begin{proof}
We will need the result established by Gauss \cite{Gauss1801} that an odd number is a square in $\mathbb{Z}_{2^k}$ if and only if it is $\equiv 1 \mod 8$. If $t \equiv 1 \mod 4$, then either $t \equiv 1 \mod 8$ in which case it is a square in $\mathbb{Z}_{2^k}$ or $t \equiv 5 \mod 8$ in which case $t - 4$ is a square in $\mathbb{Z}_{2^k}$ and $t = 4$ plus a square. In both cases $t \in S_{2^k}$. Now \mbox{$2^s \in S_{2^k}$} for any $s$ as it is either a square if $s$ is even or the sum of $2$ even powers of $2$ if $s$ is odd. Therefore, $t \cdot 2^s \in S_{2^k}$ for all $s$ when $t \equiv 1 \mod 4$ by lemma \ref{monoid}.
Next, let $t \equiv 3 \mod 4$ and suppose $n = t \cdot 2^s \in S_{2^k}$. We can assume that $s \leq k-2$. By Gauss' result, the sum of two squares in $\mathbb{Z}_{2^k}$ is $\equiv \{ \, 0, 1, 2, 4, 5 \, \} \mod 8$. Also $n \equiv \{ \, 0, 3, 4, 6, 7 \, \} \mod 8$. For equality to hold we must have both sides to be $\, \, \equiv 0 \mod 4$ which implies
$$
t \cdot 2^s \equiv x^2 + y^2 \mod 2^k
$$
where both $x^2$ and $y^2$ are $\, \equiv 0 \mod 4$. Then, dividing by $4$, we have
$$
t \cdot 2^{s-2} \equiv x_1^2 + y_1^2 \mod 2^{k-2}
$$
for some numbers $x_1$ and $y_1$. The same argument can be applied again and we descend to one of two possible outcomes. Either
$$
t \equiv a^2 + b^2 \mod 2^{k-s}
$$
or
$$
2 \cdot t \equiv a^2 + b^2 \mod 2^{k+1-s}
$$
for some $a$ and $b$. The first case is impossible as $2^{k-s} \geq 4$ and the sum of two squares $\mod 4$ cannot be $\equiv 3 \mod 4$. The second case is also impossible as $2^{k+1-s} \geq 8$ and the sum of two squares $\mod 8$ cannot be $\, \equiv 6 \mod 8$. Therefore $t \cdot 2^s \in N(2^k)$.
\end{proof}

\bigskip

\begin{theorem}
\label{1mod4}
Let $p$ be a prime with $p \equiv 1 \mod 4$ and $k \geq 1$. Then
$$
S_{p^k} = \mathbb{Z}_{p^k} \, \, \, \text{and} \, \, \, N_{p^k} = \emptyset .
$$
\end{theorem}
\begin{proof}
We first use a pigeon-hole argument so show that $\{1, 2, ... , p-1 \} \in S_{p}$. Let $x \in \{1, 2, ... , p-1 \} $ and consider the two sets
$$
A = \{\, x - y^2 \mod p: y = 0, 1, 2, ... , \frac{p-1}{2} \, \} 
$$
and
$$
B = \{\, z^2 \mod p: z = 0, 1, 2, ... , \frac{p-1}{2} \, \} .
$$
Both $A$ and $B$ contain $\frac{p+1}{2}$ distinct elements. Since  $\, \abs{\mathbb{Z}_p} = p$, $A$ and $B$ must have a common element. So there is a $y$ and $z$ such that $x - y^2 = z^2 \mod p$ and $x \in S_{p}$. We can now use Hensel's lemma to show that any positive integer relatively prime to $p$ can be written as a sum of squares $\mod p^k$. For, if $x$ is relatively prime to $p$ then there exists a $y$ such that the polynomial
$$
f(z) = x - y^2 - z^2
$$
has a simple root mod $p$. By Hensel's lemma it then has a root in $\mathbb{Z}_{p^k}$. We next observe that since $p \equiv 1 \mod 4$, $p$ can be written as a sum of squares and by lemma~\ref{monoid} so can every power of $p$. Any element in $\mathbb{Z}_{p^k}$ can be written as $x \cdot p^s$ for some $s \geq 0$ and $x$ relatively prime to $p$. Lemma~\ref{monoid} then implies that all such numbers can be written as a sum of squares $\mod p^k$. 
\end{proof}

\bigskip
In order to treat the case where $p \equiv 3 \mod 4$ we need the following lemma.

\bigskip

\begin{lemma}
\label{3mod4lemma}
If $x^2 + y^2 \equiv 0 \mod p$ where $p \equiv 3 \mod 4$ then $x \equiv 0 \mod p$ and $y \equiv 0 \mod p$
\end{lemma}
\begin{proof}
We may assume that $0 \leq x, y  \leq \frac{p-1}{2}$ by replacing $x$ with $x \mod p$ and then $p-x$ if necessary and the same for y. Then $x^2 + y^2 = c \cdot p$ where $c \leq \frac{p-1}{2}$. But the power of $p$ in $c \cdot p$ is odd so $c \cdot p$ cannot be represented as a sum of squares (see section \ref{intro}). Therefore both $x$ and $y$ must be $0$. 
\end{proof}

\bigskip

\begin{theorem}
\label{3mod4}
Let $p$ be a prime with $p \equiv 3 \mod 4$ and $k \geq 1$. Then
$$
S_{p^k} = \{ 0 \} \cup \{ t \cdot p^{2s} : s \in \mathbb{N}, \, \, t \ne 0 \mod p \}
$$
and
$$
N_{p^k} = \{ t \cdot p^{2s + 1} : s \in \mathbb{N}, \, \, t  \ne 0 \mod p \}.
$$
\end{theorem}
\begin{proof}
Firstly, the pigeon-hole/descent argument from theorem~\ref{1mod4} can be used again to show that all numbers relatively prime to $p$ are in $S_{p^k}$. Next we look at multiples of $p$. Even powers of $p$ (including $1$) are squares and so are in $S_{p^k}$. Therefore lemma~\ref{monoid} implies that all numbers of the form $t~\cdot~p^{2s}$ where $s \in \mathbb{N}$ and $t$ is relatively prime to $p$ are in $S_{p^k}$. Let $n = t~\cdot~p^{2s+1}$ where $s \in \mathbb{N}$ with $2s+1 \leq k-1$ and t is relatively prime to $p$. If $n \in S_{p^k}$ then 
$$
t \cdot p^{2s + 1} = x^2 + y^2 + c \cdot p^{k}  
$$
for some $x, y, c$. We deduce that $x^2 + y^2 \equiv 0 \mod p$ and lemma~\ref{3mod4lemma} then implies that both $x$ and $y$ are $\equiv 0 \mod p$. Dividing through by $p^2$ we have
$$
t \cdot p^{2s - 1} = x_1^2 + y_1^2 + c \cdot p^{k-2}.  
$$
Applying the same argument repeatedly we descend to the condition
$$
t \cdot p = a^2 + b^2 + c \cdot p^{k-2s}  
$$
for some $a$ and $b$. Again $a$ and $b$ are $\equiv 0 \mod p$ but this creates an impossibility as the right hand side is divisible by $p^2$ as $k-2s \geq 2$ but the left hand side is only divisible by $p$. We deduce that all numbers of the form $t \cdot p^{2s + 1}$ are in $N_{p^k}$.
\end{proof}

\bigskip

\begin{corollary}
\label{3modcor}
Let $p \equiv 3 \mod 4$ be prime and $k \geq 1$. Then
$$
\abs{N_{p^k}} = 
\begin{cases}
    \frac{1}{p+1} (\, p^k - 1 \, ) , & \text{if } \, \, k \text{ is even} \\
    \frac{1}{p+1} (\, p^k - p \, ),   & \text{if} \, \, k \text{ is odd}.
\end{cases}
$$
\end{corollary}
\begin{proof}
We use theorem \ref{3mod4} to count the elements of $N_{p^k}$. We have
\begin{align*}
\abs{N_{p^k}}  & =  \sum_{s=0}^{\lfloor \frac{k-2}{2} \rfloor}  \abs{ \{ t: t \cdot p^{2s+1} < p^k, t \ne 0 \mod p \} } \\
= & \sum_{s=0}^{\lfloor \frac{k-2}{2} \rfloor }  \abs{ \{ t: t < p^{k - 2s - 1}, t \ne 0 \mod p \} } \\
= & \sum_{s=0}^{\lfloor \frac{k-2}{2} \rfloor } \abs{    \mathbb{Z}_{p^{k-2s-1}} ^* }
\end{align*}
where  $\mathbb{Z}_{p^{k-2s-1}} ^*$ is the set of invertible elements in $\mathbb{Z}_{p^{k-2s-1}}$.  Since $\abs{  \mathbb{Z}_{p^r} ^*} = p^r - p^{r-1}$, when $k$ is even the above sum  reduces to
$$
\sum_{s=0}^{\frac{k-2}{2}} (p^{2s + 1} - p^{2 s} ) 
$$
$$
=  (p-1) \sum_{s=0}^{\frac{k-2}{2}} (p^{2s} )		
$$
$$
= (p - 1) \frac{ (p^2)^{k/2} - 1}{p^2 - 1} 
$$
$$
= \frac{1}{p+1} (p^k -1).
$$
The calculation for $k$ odd is similar.
 \end{proof}

\bigskip

We can now look at the general situation. The following theorem was proved in \cite{Harrington:2014aa} (theorem 4.1).

\bigskip

\begin{theorem}
Let $ \, \, n \in \mathbb{N}$. Then $ \, \, S_n = \mathbb{Z}_n \, \, $ if and only if the following condition holds:
$$
\mbox{if $ \, n \equiv 0 \mod p^2 \, $ where $ \, p$ is prime, then $ \, p \equiv 1 \mod 4$}.
$$
\end{theorem}
\begin{proof}
The positive integer $n$ can be written as a product of primes
$$
n = \prod_{p_i} p_i^{\alpha_i}.
$$
and so from corollary \ref{cor}
$$
r(S_n) = \prod_{p_i} r(S_{p_i^{\alpha_i}}).
$$
Now $S_n = \mathbb{Z}_n \, \, $ if and only if $r(S_n) = 1$. Therefore we must have $r(S_{p_i^{\alpha_i}}) = 1$ for each prime $p_i$ dividing $n$. From theorems \ref{2k}, \ref{1mod4} and \ref{3mod4} we have that $r(S_{2^i}) <  1$ for $i  > 1$ and $r(S_{p^i}) < 1$ for primes $p \equiv 3 \mod 4$ when $ i > 1$. We also have that $r(S_{p^i}) = 1$ for all primes $p \equiv 1\mod 4$ and all $i$. The result follows.
\end{proof}

\bigskip

We note that Harrington, Jones and Lamarche \cite{Harrington:2014aa} also investigated which positive integers $z$ can be written as the sum of two {\bf nonzero} squares modulo a positive integer $n$. They obtained a complete characterisation of those $n$ such that every $z \in \mathbb{Z}_n$ can be written as the sum of two nonzero squares in $\mathbb{Z}_n$.

\bigskip

\section{Some properties of $S_n$ and $N_n$}
\label{properties}

Since $\abs{S_n} = n - \abs{N_n}$, the properties of $S_n$ and $N_n$ are connected. We investigate the ratios $r(S_n)$ and $r(N_n)$ which satisfy $r(S_n) = 1 - r(N_n)$. The limit of these ratios as $n \to \infty$ does not exist. In fact we have:
\bigskip

\begin{theorem}
$$\liminf_{n \to \infty} r(S_n) = \liminf_{n \to \infty} r(N_n) = 0
$$
and
$$
\limsup_{n \to \infty} r(S_n) = \limsup_{n \to \infty} r(N_n) = 1.
$$
\end{theorem}
\begin{proof}
To show $\limsup_{n \to \infty} r(S_n) = 1$ we take $n_k = 5^k$. Then, since $n_k \equiv 1 \mod 4$, $r(S_{n_k}) = 1$ for all $k$.
To show $\limsup_{n \to \infty} r(N_n) = 1$, we consider primes that are congruent to $\, 3~\mod~4$. Let $p_k$ be the $k \, $th prime that is congruent to $ \,\, 3 \mod 4$ and form the product 
$$
n(i, s) = \prod_{k = 1}^{i} p_k^s
$$
where $s \in \mathbb{N}$. Then since \mbox{$\,\, \lim_{s \to \infty} r(N_{p_k^s}) = \frac{1}{1+p_k} \,\, $} from corollary \ref{3modcor}, we have
$$
\lim_{s \to \infty} r(N_{n(i, s)}) = 1 - \prod_{k = 1}^{i} (1 - \frac{1}{1+p_k}) 
$$ 
from corollary \ref{cor}. Therefore
$$
\limsup_{n \to \infty} r(N_n) \geq 1 - \prod_{k = 1}^{\infty} (1 - \frac{1}{1+p_k}).
$$	
The later product is $0$ as its reciprocal is
$$
\prod_{p \equiv 3 \mod 4} (1 + \frac{1}{p})
$$
which diverges by comparison with 
$$
\sum_{p \equiv 3 \mod 4} \frac{1}{p}
$$
by the strong form of Dirichlet's theorem.
\end{proof}

\bigskip

\section{Sums of two squares and powers of $2$ in $\mathbb{Z}_n$}
\label{powers}

Crocker \cite{Crocker:2008id} investigated the representation of integers as the sum of two squares and up to two non-negative integral powers of $2$. He first showed that if $t \equiv 0 \mod 36$ and $t$ cannot be written as the sum of two squares and up to two powers of $2$, then neither can $2^{\alpha} \cdot t$ for any $\alpha \in \mathbb{N}$. He then created a system of congruences to show that there are at least $94$ numbers $t < 2^{1417}$ which are $\,\, \equiv 0 \mod 36$ and which cannot be written as the sum of two squares and up to two powers of $2$. As a result, it follows that the asymptotic density of numbers which cannot be written as the sum of two squares and up to two powers of two is bounded below by \mbox{$\,\, 94 \cdot \frac{\log n}{n}$}. In fact the constant $94$ can be made larger as stated in \cite{Crocker:2017}.

Using some of Crocker's ideas, Platt and Trudgian \cite{Platt:2016aa} also studied the problem and showed that if $\, \, N_0 = 1 151 121 374 334 \,\,$ and $\,\, \alpha \geq 0 \,\,$, no number of the form $2^{\alpha} \cdot N_0$ can be written as the sum of $2$ squares and up to two powers of $2$. They also showed that the smallest positive integer which cannot be written as a sum of two squares and up to two powers of 2 is 535903.

In \cite{Platt:2016aa} Platt and Trudgian also discussed numbers which can be written as the sum of two squares and up to one power of $2$. At the end of their paper they asked whether it is possible to obtain good estimates on the density of numbers that can and cannot be represented as the sum of two squares and up to one power of $2$. They showed that numbers of the form $23 + 72k; \, \, k \in \mathbb{N}$ cannot be written in that form. The density of this set of numbers is therefore bounded below by $\frac{1}{72}$. By working with modular arithmetic in a similar way to Crocker and Platt and Trudgian we can show that a larger density of numbers are not able to be written as a the sum of two squares and up to one power of $2$. 

Following the idea above we work modulo $n = 2^k \cdot m$, where $m$ is odd, and try to choose $k$ and $m$ to maximise the density of the set $A$ of numbers $x$ in $\mathbb{Z}_n$ satisfying the following conditions

\bigskip

\begin{enumerate}
\item $x \mod 2^k \in N_{2^k}$
\item $x - 2^i \in N_n \,\, $ for $i = 0, 1, ... , k-1$
\end{enumerate}

\bigskip

The first condition ensures that numbers of the form $x \mod n$ and $x - 2^j \mod n$ for $j \geq k$ cannot be written as a sum of squares in $\mathbb{Z}$ as they are in $N_{2^k}$. The second condition ensures that numbers of the form $x - 2^i \mod n$ for $i = 0, 1, ... , k-1$ cannot be written as a sum of squares in $\mathbb{Z}$ as they are in $N_n$. When $x$ satisfies both conditions we know that numbers congruent to $x \mod n$ cannot be written as the sum of two squares or as the sum of two squares and a power of $2$ in $\mathbb{Z}$. Therefore, at least $\frac{\abs{A}}{n}$ of numbers cannot be written as a sum of squares and up to one power of $2$ in $\mathbb{Z}$. We would like to choose $k$ and $m$ to maximise $\frac{\abs{A}}{n}$. For any $k$ we know from theorem \ref{2k} that the number of $x \in \mathbb{Z}_{2^k}$ satisfying the first condition is $2^{k-1}-1$ so the number of $x \in \mathbb{Z}_n$ satisfying the first condition is $m \cdot (2^{k-1}-1)$. The density of such numbers in $Z_n$ is therefore $\frac{1}{2} - 2^{-k}$. The most that this approach can do then is to show that at least \mbox{$\frac{1}{2} - \epsilon$} of numbers cannot be written as a sum of two squares and up to one power of $2$ (if in fact the actual density is that high). There does not seem to be any way to connect the choice of $k$ in the first condition with the choice of $m$ to maximise the density of numbers satisfying the second condition as well. One possible approach is to choose $k$ and $m$ to maximise the density of $N_n$ in $\mathbb{Z}_n$ in the hope that this will make it more likely that the $k$ numbers $x - 2^i$ hit $N_n$ for each $x$ satisfying the first condition. In choosing $m$ we only need to consider numbers which are not divisible by any primes $\,\, \equiv 1 \mod 4$. This is because primes that are $\,\, \equiv 1 \mod 4 \,\,$ do not change the density of $N_n$, i.e. if $p \equiv 1 \mod 4$ then the density of $N_{pn}$ in $\mathbb{Z}_{pn}$ is the same as the density of $N_n$ in $\mathbb{Z}_n$ by equation (\ref{rmns}) and theorem \ref{1mod4}. We also only need to consider values of $m$ which are square-full, i.e. if a prime $p$ divides $m$ then $p^2$ also divides $m$. This is because $N_p =  \emptyset$ for all primes $p$. For primes $p \equiv 3 \mod 4$ we have
$$
r(N_{p^{\alpha}}) = \frac{1}{p+1} ( \,\, 1 - c p^{-\alpha} \, \, )   \sim  \frac{1}{p+1} \,\, \text{ for large } \,\, \alpha
$$
where $c$ is either $1$ or $p$. $r(N_{p^{\alpha}})$ increases as $\alpha$ increases but reduces as $p$ increases and approaches a maximum of $\frac{1}{p+1}$ as $\,\, \alpha \to \infty$. A reasonable search method for finding good values of $k$ and $m$ is to start with values of $m$ which are products of small primes $\equiv 3 \mod 4$ raised to the power $2$ and gradually raise the powers of the primes and add more primes. Since we are working in $\mathbb{Z}_n$ we can use theorems \ref{2k}, \ref{1mod4} and \ref{3mod4} to avoid having to calculate the elements of $N_{2^k}$ and $N_n$ manually. This of course saves a good deal of computing time. Our approach instead was to perform a linear search over values of $n$ up to $10$ million, skipping those $n$ divisible by a prime congruent to $1 \mod 4$ or by any prime $p$ such that $p^2  \nmid n$. In this range, the values of $k$ and $m$ which maximise the density of numbers in $\mathbb{Z}_n$ satisfying the two conditions above are $k = 7$ and $m = 3^2 \cdot 7^2 \cdot 11^2 = 53361$ so $n = 6830208$. For this value of $n$ there are $828139$ numbers $\mod n$ which satisfy both conditions. For each of these $828139$ numbers $x$ we know that numbers of the form $x \mod 6830208$ cannot be written as a sum of squares plus up to one power of $2$. Therefore, {\bf at least $12.125\%$ of numbers cannot be written as a sum of squares plus up to one power of $2$.}

The same approach can be used when investigating which numbers can be represented by a sum of squares and up to two powers of $2$. The two conditions mentioned above need to be adjusted to the new situation. Let $n = 2^k \cdot m$ where $m$ is odd.We aim to choose $k$ and $m$ to maximise the density of the set $A$ of numbers $x$ in $\mathbb{Z}_n$ satisfying the following conditions

\bigskip

\begin{enumerate}
\item $x \mod 2^k \in N_{2^k}$
\item $x - 2^i \in N_n \,\, $ for $i = 0, 1, ... , k-1$
\item $x - 2^i - 2^j  \in N_n \,\, $ for $i = 0, 1, ... , k-1 \,\, $ and $j = i+1, i+2, ... , ord_{n/2^k} (\, 2 \, ).$
\end{enumerate}

\bigskip

Here, $\, ord_{n/2^k} (\, 2 \, ) \,$ is the order of the integer $2$ in the group $\, \mathbb{Z}^*_{n/2^k} \,$. The first two conditions again ensure that numbers of the form $x \mod n$ cannot be written as the sum of two squares or as the sum of two squares and a power of $2$ in $\mathbb{Z}$. The third condition then implies $\, x \mod n$ cannot be written as the sum of two squares and two powers of $2$. The set of numbers $A$ in $\mathbb{Z}_n$ satisfying the three conditions above provides a lower bound on the density of numbers which cannot be written as the sum of two squares and up to two powers of $2$. We tested the first million values of $n$ but found that the set $A$ was empty for each of these $n$. If the asymptotic density of numbers in $\mathbb{Z}$ that cannot be written as the sum of two squares and up to two powers of $2$ is zero (which would be the case if it took the form $\BigO{\frac{\log n}{n}}$ for example) then the set $A$ satisfying the three conditions above is empty for all $n \in \mathbb{N}$.

\bigskip

\bibliographystyle{plain}
\begin{small}
\bibliography{ref}
\end{small}

\end{document}